\documentclass[preprint]{elsarticle}
\usepackage{lineno,hyperref}
\usepackage{amsmath}
\usepackage{amsthm}
\usepackage{amssymb}
\usepackage{graphicx}
\usepackage{enumitem}
\usepackage{comment}

\newtheorem{theorem}{Theorem}[section]
\newtheorem{proposition}[theorem]{Proposition}
\newtheorem{corollary}[theorem]{Corollary}
\newtheorem{definition}[theorem]{Definition}

\begin{document}

\begin{frontmatter}

\title{On graphs with adjacency and signless Laplacian matrix eigenvectors entries in $\{-1, +1\}$}


\author[Jorgeaddress]{Jorge Alencar}
\ead{jorgealencar@iftm.edu.br}
\address[Jorgeaddress]{Instituto Federal de Educa\c{c}\~{a}o, Ci\^{e}ncia e Tecnologia do Tri\^angulo Mineiro, Brazil}

\author[Leoaddress1,Leoaddress2]{Leonardo  de Lima} \ead{leonardo.delima@ufpr.br}
\address[Leoaddress1]{Department of Business, Federal University of Parana, Paran\'a, Brazil}
\address[Leoaddress2]{Systems and Production Engineering Program, Federal Center of Technological Education Celso Suckow da Fonseca, Rio de Janeiro, Brazil}

\begin{abstract}
Let $G$ be a simple graph. In 1986, Herbert Wilf asked what  kind of graphs  has an eigenvector with entries formed only by $\pm 1$? In this paper, we answer this question for the adjacency, Laplacian and signless Laplacian matrix of a graph. Besides, we generalize the concept of an exact graph to the adjacency and signless Laplacian matrices. Infinity families of exact graphs for all those matrices are presented.
\end{abstract}

\begin{keyword}
eigenvector \sep  adjacency matrix \sep signless Laplacian matrix \sep Wilf's problem \sep exact graph.
\MSC[2010] 05C50, 
05C35 
\end{keyword}

\end{frontmatter}

\section{Introduction}\label{intro}

Let $G=(V,E,w)$ be a simple weighted graph with vertex set $V$ and edge set $E$ such that  $|V|=n$ and $|E|=m.$ Given a  subset $U \subseteq V$, let $G[U]$ be a subgraph induced by $V$ with edge set $E[U].$ A weight $w_{ij}>0$ is assigned to each edge $e_{ij} = v_i v_j \in E$ such that $w_{ij}=w_{ji}$. The degree of a vertex $v \in V$, denoted by $d_i(G)$, is given by the sum of the weights of the edges incident to $v_i$. If $G$ is a simple non-weighted graph, we take  $w_{ij}=w_{ji} = 1$ if $e_{ij} \in E$. We write $A(G)$ for the adjacency matrix of $G$ where $a_{ij}=w_{ij}$ if $e_{ij} \in E$ and $a_{ij}=0$ otherwise. The diagonal matrix $D(G)$ is given by the row-sums of $A,$ i.e., the degrees of $G.$ As usual, $L(G)=D(G) - A(G)$, is the Laplacian matrix of $G$, and $Q(G)=D(G)+A(G)$, is the signless Laplacian matrix of $G$. The $A(G)$, $L(G)$ and $Q(G)$ eigenvalues, called $A$-, $L$- and $Q$-eigenvalues, are arranged as follows: $\lambda_1(G) \geq \cdots \geq \lambda_n(G),$ $\mu_1(G) \geq \cdots \geq \mu_{n-1}(G) \geq \mu_n(G) = 0$ and $q_1(G) \geq \cdots \geq q_n(G).$ 

Consider a partition of $V$ into subsets $X$ and $\overline{X} = V \setminus X$ and let $E(X,\overline{X}) = \{e_{ij} | i \in X, j \in \overline{X} \}$, i.e., the set of edges with ending points in $X$ and $\overline{X}.$ The weight of $E(X,\overline{X})$, denoted by $cut_{G}(X)$, is given by the sum of the edge weights  in $E(X,\overline{X})$, that is, $$ cut_{G}(X) = \sum_{e_{ij} \in E(X,\overline{X})} w_{ij}.$$ In this paper, we address the well-known maximum cut problem (MCP) which consists of finding an optimal bipartition $\{X,\overline{X}\}$ that maximizes $cut_{G}(X)$ and is defined as 
$$mcut(G)=  \max_{X \subseteq V} cut_G(X)$$ and its relation to spectral parameters of a graph.
Some work has been done in the literature relating the $mcut(G)$ and the eigenvalues of a graph. Mohar and Poljak \cite{16} presented an upper bound on the maximum cut of a graph based on the largest eigenvalue of its Laplacian matrix, 
\begin{equation}\label{Eq1-Lmaxcut}
mcut (G) \leq \frac{n}{4} \mu_1(G).
\end{equation}
In \cite{16}, the authors defined an \emph{exact graph} as the ones whose equality (\ref{Eq1-Lmaxcut}) holds and also showed some families of this kind of graphs. Delorme and Poljak \cite{delorme} also presented an upper bound to the maximum cut problem in terms of the Laplacian matrix plus a correcting matrix and presented some classes of exact graphs. Computing Delorme and Poljak's upper bound is equivalent to solving an eigenvalue minimization problem and it is computable in polynomial time with arbitrary precision. 
We begin the paper by presenting some bounds to $mcut(G)$ involving the smallest eigenvalues of the matrices $A(G)$ and $Q(G)$ which show a relation among those eigenvalues and the $mcut(G)$, which is analogous to the relation between the largest eigenvalue of $L(G)$ and $mcut(G)$ established in \cite{16}. Based on that, we generalize the concept of a graph exactness, which originally is related to a Laplacian eigenvalue of the graph. We extend it here by defining $A$-exact and  $Q$-exact graphs when our parameter is related to the adjacency or to the signless Laplacian matrix respectively. Infinity families of graphs $A$-, $L$- and $Q$-exact graphs are presented. Besides, we stated necessary and sufficient conditions for a graph have an eigenvector uniquely formed by entries $\pm 1.$ This result answers an intricate Wilf's question posed in \cite{wilf}, where the author asks what kind of graphs has an eigenvector uniquely formed by entries $\pm 1$ and is one of our main contribution. Stevanovi\'{c} \cite{dragan2016} mentioned that the set of graphs having the property of an eigenvector with entries $\pm 1$ is quite rich and proved that Wilf's problem is NP-Complete. Very recently Caputo, Khames and  Knippel \cite{caputo2019} extended the Wilf's question to the graph Laplacian eigenvector with entries $\pm 1$ and solved it. We propose an alternative proof to Caputo, Khames and  Knippel's proof. Also, we found a characterization of all graphs having a signless Laplacian graph eigenvector with entries $\pm 1.$ 

The paper is organized as follows. In Section \ref{sec1}, we show that the maximum cut of a graph is bounded above by the A- and Q-spread of a graph. In Section \ref{MainResult}, we present a graph characterization for an eigenvector of matrices $A$ and $Q$ having entries $\pm 1.$ Also, a generalization of the concept of an exact graph is presented. In Section \ref{infinityfamilies}, we present infinity families of $A$- and $Q$-exact graphs and conclude the paper.

\section{The maximum cut and the spread of a graph}\label{sec1}
Given a bipartition $X, Y\subseteq  V$  and $E(X,Y)$ as the set of all edges with ending points in $X$ and $Y$, the \textit{cohesion} of its partition is given by  
$$coh_G(X) = \sum _{e_{ij} \in \overline{E(X,Y)}} w_{ij}.$$ 
Observe that given a graph $G$, maximizing $cut_G(X)$ is equivalent to minimize $coh_G(X)$ since $coh_G(X)+cut_G(X)=W$, where $W = \sum_{e_{ij} \in E}w_{ij}.$ 
Thus, we define the minimum cohesion by
$$mcoh(G)=\min _{X\subseteq V}coh_G(X)$$
and immediately obtain the result of Proposition \ref{Prop1}.

\begin{proposition}\label{Prop1} Let $U$ be a subset of $V$, then $U$ maximizes $cut_G(U)$ if and only if minimizes the $coh_G(U)$.
\end{proposition}
%
From the signless Laplacian theory, or $Q-theory$, a basic result states that for $x \ne 0 \in R^{n}:$
\begin{equation}\label{sl}
\langle Q(G)x, x \rangle = \sum_{ij \in E} w_{ij}(x_i+x_j)^2.
\end{equation}
Let $S \subset V$. We define a partition vector $p_S \in \mathcal {R}^n$ such that 
$$p_S^{i}=\left\lbrace \begin{array}{l}
1,\hspace{0.2cm} v_i \in S\\
-1,\hspace{0.2cm} v_i \notin S.
\end{array}\right. $$

Thus, from the partition vector $p_S$ and equation (\ref{sl}), since $p^{i}_S + p^{j}_S$ is equal to 2   
whenever $ij \in E[S]$ and $0$ otherwise, we get
\begin{equation}\label{sl2}
\langle Q(G)p_S, \;p_S \rangle = \sum_{ij \in E} w_{ij}(p^{i}_S+p^{j}_S)^2 = 4 \; coh_G(S).
\end{equation}

Using the latter facts, we easily get an upper bound to $mcut(G)$ of a weighted graph, the same result obtained by de Lima \emph{et al.} (2011) in \cite{nik}. 

\begin{proposition}\label{Prop3} Let $G=(V,E)$ be a graph with $n$ vertices, $W$ be the sum of all edge weights of $G$ and $q_n(G)$ be the smallest eigenvalue of $Q$. Then,
\begin{equation}\label{Prop3ineq}
mcut(G) \leq W - \frac{nq_n(G)}{4}.
\end{equation}  
\end{proposition}

\begin{proof}Let $p_S$ be the partition vector related to $S\subset V$ and let $\mathbf{x}$ be an eigenvector to $q_{min}(G).$  Then, from the Rayleigh principle we have
\begin{eqnarray}
q _n (G) &=& \min _{x\neq 0}\frac{x^tQx}{x^tx} \leq \min _{S\subseteq V}\frac{p_S^tQp_S}{n} = \frac{1}{n} \min_{S \subseteq V}\sum_{ij \in E(G)} w_{ij}(p^i_{S} + p^j_{S})^2 \nonumber \\ 
&\leq & \frac{4}{n}mcoh(G) = \frac{4}{n}(W- mcut(G)) \nonumber
\end{eqnarray}
and the result follows.
\end{proof}

A similar relation can be obtained for the smallest eigenvalue of the adjacency matrix of $G.$ 

\begin{proposition}\label{Prop2}Let $G=(V,E)$ be a weighted graph on $n$ vertices, $W$ be the sum of weights of all edges of $G$ and $\lambda_n(G)$ be the smallest eigenvalue of $A$. Then,
\begin{equation}
mcut(G) \leq \frac{W}{2} - \frac{n \lambda_{n}(G)}{4}.
\label{Prop2ineq}
\end{equation}
\end{proposition}
\begin{proof} Let $p_S$ be a partition vector of the set $S\subset V$, then for the Rayleigh-Ritz theorem we have
\begin{eqnarray}
\lambda _n(G) & = & \min_{x\neq 0}\frac{x^tAx}{x^tx} \leq \min_{S\subseteq V} \frac{p_S^{T} A p_S}{n}  = \min_{S\subseteq V} \left( \frac{p_S^T Q p_S}{n}-\frac{p_S^T D p_S}{n} \right) \nonumber \\ 
		       & = & \min_{S\subseteq V} \left( \frac{p_S^tQp_S}{n}-\frac{2}{n} W \right) = \min_{S\subseteq V} \left( \frac{p_S^tQp_S}{n} \right) -\frac{2}{n}W \nonumber \\
		       & = & \frac{4}{n}mcoh(G)  -\frac{2}{n}W \nonumber \\
		       & = & \frac{2}{n} \left(W-2 mcut(G)\right) \nonumber
\end{eqnarray}
and the inequality follows.
\end{proof}

The bounds (\ref{Prop3ineq}) and (\ref{Prop2ineq}) are the best possible and graphs attaining their equalities are here called $Q-$exact and $A-$exact, respectively. It is clear that complete graphs of even order and bipartite graphs are $Q-$ exact. While describing all equality cases seems to be a difficult task, we present infinity families of $A$- and $Q$-exact graphs in Section \ref{infinityfamilies}. 

The spread of a Hermitian matrix $M$ with eigenvalues $\lambda_1(M) \geq \cdots \geq \lambda_{n}(M)$ is defined as $s(M) = \lambda_{1}(M)-\lambda_{n}(M).$ In particular, the spread of the matrices $A, L$ and $Q$ related to graphs are defined, respectively, by 
$$s_{A}(G) = \lambda_{1}(G)-\lambda_{n}(G),$$ $$s_{L}(G) = \mu_{1}(G) $$ and $$s_{Q}(G) = q_{1}(G)-q_{n}(G).$$  

Since $\lambda_{1}(G) \geq 2W/n$ and $q_{1}(G) \geq 4W/n$, from Propositions \ref{Prop1} and \ref{Prop2} the next result is straightforward. 

\begin{corollary} \label{Teorema1} Let $G=(V,E)$ be a weighted graph on $n$ vertices and let $N$ be either the adjacency, $A(G)$, or the signless Laplacian, $Q(G)$, of $G$. Then,
\begin{equation}\label{Eq2-Nmaxcut}
 mcut (G) \leq \frac{n}{4} s_{N}(G).
\end{equation}

\end{corollary}

%

Since inequalities (\ref{Eq1-Lmaxcut}) and (\ref{Eq2-Nmaxcut}) have the same lower bound for any graph $G$ one may think about the relation between the spreads of the matrices $A, L$ and $Q.$ In order to find those relations, we use two basic results: one from Merikoski and Kumar \cite{15}, where the authors proved that for any two Hermitian matrices of the same order, say $M$ and $N$, $s(M+N)\leq s(M) + s(N);$ the second is the well-known Weyl's inequality presented here with their equality conditions in Theorem \ref{th:weyl}.




\begin{theorem}\label{th:weyl} 
Let $A$ and $B$ be Hermitian matrices of order
$n,$ and let $1\leq i\leq n$ and $1\leq j\leq n.$ Then
\begin{eqnarray}
\lambda_{i}(A)+\lambda_{j}(B)  &  \leq\lambda_{i+j-n}(A+B), \mbox{ if  } i+j\geq
n+1,\label{Wein1}\\
\lambda_{i}(A)+\lambda_{j}(B)  &  \geq\lambda_{i+j-1}(A+B),\mbox{ if  }i+j\leq
n+1. \label{Wein2}%
\end{eqnarray}
In either of these inequalities equality holds if and only if there exists a
nonzero $n$-vector that is an eigenvector to each of the three involved eigenvalues.
\end{theorem}

Now, we are able to prove a relation between the spreads.

\begin{theorem}\label{Prop7} Let $G$ be a connected graph. Then, 
\begin{equation}
2s_{A}(G)\leq s_{L}(G)+s_{Q}(G).\label{Prop6ine}
\end{equation}
Equality holds if and only if $G$ is regular.
\end{theorem}
\begin{proof}
Using Merikoski and Kumar's result for the matrix $2A=Q-L$, we find that  $2s_{A}(G)\leq s_{L}(G)+s_{Q}(G)$, since $s_{L}(G) = s_{-L}(G).$
Now, we need to prove the equality case. Suppose that $2s(A)=s(L)+s(Q)$ and $G$ is non-regular. In \cite{chen} the author states that the equality cases of
$q_1 (G) \geq 2 \lambda_1(G)$
are restricted to regular graphs. By hypothesis, $G$ is non-regular so $q_1(G)> 2 \lambda_1(G)$. Besides, using inequality (\ref{Wein1}) to $2A=Q-L$ we obtain 
$$2\lambda_n(G) \geq q_n(G)-\mu_1(G).$$ 
Taking the difference of the latter inequalities we obtain $$s_Q(G)+s_L(G) > 2s_A(G)= s_Q(G)+s_L(G),$$
which is a contradiction. Therefore, $G$ is regular. 
Now, suppose that $G$ is $r-$regular. It is known that $\lambda_{i}(G) = q_{i}(G) - r = r - \mu_{i}$ for $i=1,\ldots,n$ and the result follows.
\end{proof}

Notice that if $G$ is disconnected, the equality case of Proposition \ref{Prop7} holds for non-regular graphs, see for instance when $G = K_{4}+C_{4}.$

\section{Main results}\label{MainResult}

In \cite{16},  Mohar and Poljak define an exact graph as the ones for which  $\mu_1(G)=\frac{4}{n}mcut(G)$ and proved the following facts: (i) the cartesian product of $L$-exact graphs is a $L$-exact graph; (ii) the complement of the cartesian product of complete graphs $K_n$ and $K_m$ is a $L$-exact graph, if $n\leq m$ and $n$ is even; (iii) bipartite regular graphs are $L$-exact graphs; (iv) line graphs of bipartite graphs and their complements are $L$-exact graphs; (v) line graph of bipartite $(r,s)$-semiregular graph is a $L$-exact graph if $s$ and $r$ are even. Taking into account the inequalities of Propositions \ref{Prop3} and \ref{Prop2}, we introduce an extension of the exact graph definition by considering the adjacency and the signless Laplacian matrices. 

\begin{definition} \textit{We say that a graph $G$ is a:
\begin{enumerate}[label=\bfseries (\alph*)]
\item $L$-exact graph if and only if $\mu_1(G)=\frac{4}{n}mcut(G)$;
\item $Q$-exact graph if and only if $q_n(G)=\frac{4}{n}(W-mcut(G))$;
\item $A$-exact graph if and only if $\lambda _n(G) = \frac{2}{n}(W-2 mcut(G))$.
\end{enumerate}}
\end{definition}

Notice that if $G$ is $Q-$exact or $A-$exact, then equality holds in (\ref{Eq2-Nmaxcut}) if and only if $G$ is regular. Also, if $G$ is $r-$regular and $L-$exact, we immediately get that $G$ is $Q-$ and $A-$exact since $\mu_1 = r-\lambda_n$ and $\mu_1 = 2r - q_{n}.$ However, the reciprocal is not true. See for instance, the graphs of the Figure \ref{Figure1}. The graph of Figure \ref{Figure1}(a) is $A-$exact but it is neither $Q$- nor $L$-exact. The graph of Figure \ref{Figure1}(b) is $Q$-exact, but it is neither $A-$ nor $L-$exact.

\begin{figure}[h]
\begin{minipage}[c]{0.45\linewidth}
      \centering
      \includegraphics[width=\textwidth]{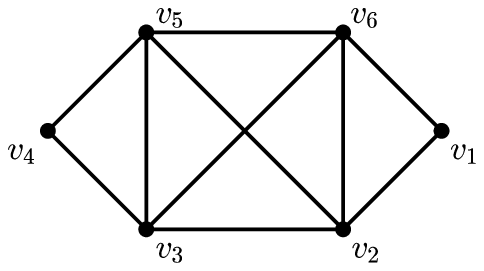}
     (a) $A-$exact graph
\end{minipage}
\begin{minipage}[c]{0.45\linewidth}
\centering
      \includegraphics[width=\textwidth]{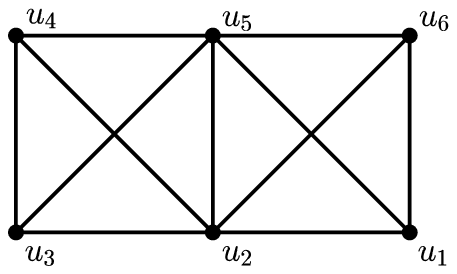}
      (b) $Q-$exact graph
\end{minipage}
\caption{The graphs in (a) and (b) are exclusively $A-$ and $Q-$exact, respectively.}
\label{Figure1}
\end{figure}

Next, Theorem \ref{Prop14} states necessary and sufficient conditions to a graph to have a vector partition $p_{S}$ for a given partition $S$ and $V\setminus S$ of a graph $G$ such that $|S|=n_1$ and $|\overline{S}| = |V \setminus S| = n_2.$ With this result, it is possible to build  infinity families of exact graphs. 

At this point, let us introduce some notation:  $e_{n}$ is the all ones $n-$vector; $A_{X} = A(G[X])$ is the adjacency matrix of the induced subgraph $G[X]$, where $X \subseteq V(G)$; $G[E(X,Y)]$ is the edge-induced subgraph by the edges with endpoints in the vertex sets $X$ and $Y;$ $D_{X}$ is the diagonal matrix of the vertices degrees of $G$ for all vertices of set $X$, i.e., $(D_{X})_{i} = d_{i}(G)$ for every $v_{i} \in X;$  $d_{i}(X,\overline{X})$ is the number of edges from vertex $v_i$ to all vertices of set $\overline{X}.$ The following result states necessary and sufficient conditions to a partition vector be an eigenvector of the matrices $L$ and $Q$ of a given graph and it was very useful to identify $Q$- and $L$-exact graphs. Also, this answers a question posed by Wilf in \cite{wilf}: ``what kind of graphs have eigenvector with entries $\pm 1$ 
only?'' We were able to answer this question for matrices $A, L$ and $Q$ of graphs in Theorem \ref{Prop14}. Notice that in case (ii) we propose an alternative proof to the one presented by \cite{caputo2019} in Theorem 13.

\begin{theorem}\label{Prop14} Let $G$ be a graph on $n$ vertices and let $p_S$ the partition vector associated to $S\subset V.$ Then, 
\begin{enumerate}
\item \label{Prop14-1} $p_S$ is an eigenvector of $Q$ associated to the eigenvalue $q$ if and only if $G[S]$ and $G[\overline{S}]$ are both  $\frac{q}{2}$-regular;
\item \label{Prop14-2} $p_S$ is a eigenvector of $L$ associated to the eigenvalue $\mu$ if and only if $G[E(S,\bar{S})]$ is $\frac{\mu}{2}$-regular;
\item \label{Prop14-3} $p_S$ is a eigenvector of $A$ associated to the eigenvalue $\lambda$ if and only if 
$d_{i}(S)-d_{i}(S, \overline{S}) =\lambda$ for all vertex $v_{i} \in V.$ 
\end{enumerate}
\end{theorem}
\begin{proof}
Let $p_S=[e_{n_1}^t,-e_{n_2}^t]^t$ be the partition vector associated to $S,$ where $n_1=|S|$ and $n_2 = |\overline{S}|.$ Write $B$ for the matrix with entries  $b_{ij} = 1$ if $\{v_i,v_j\} \in E(S,\overline{S})$ and $b_{ij}= 0,$ otherwise. 
First, suppose that $p_{S}$ is an eigenvector of $G$ associated to the Q-eigenvalue $q.$ So, 
\begin{eqnarray*}
Qp_S  =  (A+D)p_S   
	  &=&  \left( 
	 \begin{array}{cc}
	 A_{S} + D_{S}&B\\
	 B^{T}&A_{\overline{S}} + D_{\overline{S}}\\
	 \end{array}\right)
	 \left(  
	 \begin{array}{c}
	 e_{n_1}\\
	 -e_{n_2}\\
	 \end{array}\right)	    	 
	  \nonumber \\
	 & = &   \left(  
	 \begin{array}{c}
	 D_{S}\; e_{n_1} + A_{S}\; e_{n_1}-A_{\overline{S}} \; e_{n_2}\\
	 -D_{\overline{S}}\; e_{n_2}+ B e_{n_1} - A_{\overline{S}}\; e_{n_2}\\
	 \end{array}\right)	\nonumber \\
	 & = &   \left(  
	 \begin{array}{c}
	 d_{1}(G) + d_1(S) - d_1(S, \overline{S}) \\
	 \vdots \\
	 d_{n_1}(G) + d_{n_1}(S) - d_{n_1}(S, \overline{S})\\
	 d_{n_1+1}(S, \overline{S}) - d_{n_1+1}(G) - d_{n_1+1}(\overline{S})\\
	 \vdots \\
	 d_{n_1+n_2}(S, \overline{S}) - d_{n_1+n_2}(G) - d_{n_1+n_2}(\overline{S})\\
	 \end{array}\right).	\nonumber \\
\end{eqnarray*}
Notice that for every $v_i \in S$, $d_{i}(G) = d_{i}(S) + d_{i}(S,\overline{S}).$ Also, for every $v_i \in \overline{S},$ $d_{i}(G) = d_{i}(\overline{S}) + d_{i}(\overline{S},S).$ Then, 

\begin{equation}
    Qp_S  =  \left(  
	 \begin{array}{c}
	 2d_1(S)\\
	 \vdots \\
	 2d_{n_1}(S)\\
	 -2d_{n_1+1}(\overline{S})\\
	 \vdots \\
	 -2d_{n_1+n_2}(\overline{S})\\
	 \end{array}\right).	\label{eq:ult1} \\
\end{equation} 

Thus, from the eigen-equation $Q p_{S} = q\; p_{S}$ and using Equation (\ref{eq:ult1}), we obtain that 
\begin{eqnarray*}
d_{i}(S) &=& q/2, \mbox{for  } \; i = 1,\ldots,n_1 \\
d_{i}(\overline{S}) &=& q/2, \mbox{for  }\; i = n_1+1,\ldots,n_1+n_2.
\end{eqnarray*}
It implies that the induced subgraphs $G[S]$ and $G[\,\overline{S}\,]$ are $q/2-$regular and the proof of (i) is complete.

Now, suppose that $p_{S}$ is an eigenvector of $G$ associated to the L-eigenvalue $\mu.$ So, 
\begin{eqnarray}
L \; p_S  =  (D-A)p_S   
	  &=&  \left( 
	 \begin{array}{cc}
	 D_{S} - A_{S} & -B\\
	 -B^{T}& D_{\overline{S}} - A_{\overline{S}} \\
	 \end{array}\right)
	 \left(  
	 \begin{array}{c}
	 e_{n_1}\\
	 -e_{n_2}\\
	 \end{array}\right)	    	 
	  \nonumber \\
& = &   \left(  
	 \begin{array}{c}
	 D_{S} e_{n_1} - A_{S} e_{n_1} + B e_{n_2} \nonumber \\
	 -B^{T}e_{n_1} - D_{\overline{S}} e_{n_2} + A_{\overline{S}}\; e_{n_2} \nonumber \\
	 \end{array}\right)	\nonumber \\
& = &   \left(  
	 \begin{array}{c}
	 d_{1}(G) - d_1(S) + d_1(S,\overline{S}) \nonumber \\
	 \vdots \\
	 d_{n_1}(G) - d_{n_1}(S) + d_{n_1}(S, \overline{S}) \nonumber \\
	 -d_{n_1+1}(\overline{S}, S) - d_{n_1+1}(G) +  d_{n_1+1}(\overline{S}) \nonumber \\
	 \vdots \\
	 -d_{n_1+n_2}(\overline{S}, S) - d_{n_1+n_2}(G) - d_{n_1+n_2}(\overline{S}) \nonumber \\
	 \end{array}\right) 
\end{eqnarray}

 Notice that for every $v_i \in S$, $d_{i}(G) = d_{i}(S) + d_{i}(S,\overline{S}).$ Also, for every $v_i \in \overline{S},$ $d_{i}(G) = d_{i}(\overline{S}) + d_{i}(\overline{S},S).$ Then,
\begin{equation}
    L \; p_S  = \left(  
	 \begin{array}{c}
	 2d_1(S,\overline{S}) \nonumber \\
	 \vdots \\
	 2d_{n_1}(S, \overline{S})\nonumber \\
	 -2d_{n_1+1}(\overline{S},S) \nonumber \\
	 \vdots \\
	 -2d_{n_1+ n_2}(\overline{S}, S)  \\
	 \end{array}\right). \label{eq:ult2} 
\end{equation}

Thus, from the eigenequation $L p_s = \mu\; p_{S},$ and using Equation (\ref{eq:ult2}), we get that 
the graph induced by the edges $E(S,\bar{S})$ is $\frac{\mu}{2}$-regular. The proof of (ii) is complete.

Now, suppose that $p_{S}$ is an eigenvector of $G$ associated to the A-eigenvalue $\lambda.$ We get that 

\begin{equation}
    A p_{S} = \left( \begin{array}{cc}
	 A_{S} & B \\
	 B^{T} & A_{\overline{S}}\\
	 \end{array}\right) 
\left(\begin{array}{c}
e_{n_1}\\
-e_{n_2}
\end{array}\right) = \left( \begin{array}{c} 
                          d_{1}(S) - d_{1}(S,\overline{S}) \\
                          \vdots \\
                          d_{n_1}(S) - d_{n_1}(S,\overline{S}) \\
                          d_{n_1+1}(\overline{S}, S) - d_{n_1+1}(\overline{S}) \\
                          \vdots \\
                          d_{n_1+n_2}(\overline{S}, S) - d_{n_1+n_2}(\overline{S})
                     \end{array} \right). \label{eq:ult3}
\end{equation}

From the eigenequation $A p_S = \lambda p_{S}$ and using Equation (\ref{eq:ult3}), we get that 
$d_{i}(S)-d_{i}(S, \overline{S}) =\lambda$ for all vertex $v_{i} \in S$ and $d_{i}(\overline{S})-d_{i}(S, \overline{S}) =\lambda$ for all vertex $v_{i} \in \overline{S}.$ It completes the proof of the theorem.
\end{proof}

Next, we present examples of graphs and some of their subgraphs that attain Theorem \ref{Prop14}.

\vspace{0.2cm}

\noindent \textbf{Example 1} Let $G$ be the graph displayed in  Figure \ref{Figure5} and let $S=\{v_1,v_2,v_3,v_4\}.$ Denote by  $e_{ij} = \{v_i,\, v_j\}$ the edges in $G[S]$ and by $f_{ij} = \{u_{i}, \, u_{j}\}$ the edges of $G[\overline{S}].$
\begin{figure}[h]
\centering
\includegraphics[scale=1]{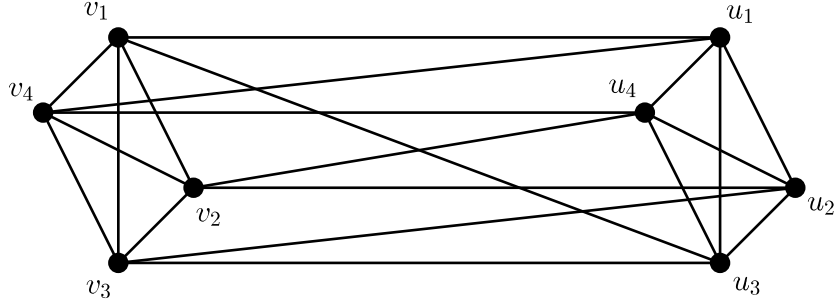}
\caption{Graph $G$ with one $L$- and $Q$-eigenvalue associated to an eigenvector with all entries equal to $\pm 1.$}
\label{Figure5}
\end{figure}
Note that $G[S]$ and $G[\overline{S}]$ are 3-regular. By Theorem \ref{Prop14}, $G$ has $p_{S}^{v_i}=1$ and $p_{S}^{u_i}=-1$ for all $i=1,2,3,4$ as an eigenvector associated to the $Q$-eigenvalue $q = 6.$ Removal of edges from $E(S,\overline{S})$, except one, does not change the $Q$-eigenvalue 6 and its corresponding eigenvector. Further, as $G[E(S,\overline{S})]$ is 2-regular, $\mu = 4$ is an $L$-eigenvalue of $G$ with the same eigenvector $p_{S}^{v_i}=1$ and $p_{S}^{u_i}=-1.$ The removal of edges from the induced subgraphs $G[S]$ or $G[\overline{S}]$ does not change $\mu=4$ and its corresponding eigenvector.

\vspace{0.3cm}

\noindent \textbf{Example 2:} Let $G$ be the graph displayed in Figure \ref{Figure6}. Let 
$V=\{v_1,v_2,v_3,v_4,\\v_5,v_6\}$ and  $U=\{u_1,u_2,u_3,u_4,u_5,u_6\}$. 
\begin{figure}[h]
\centering
\includegraphics[scale=0.15]{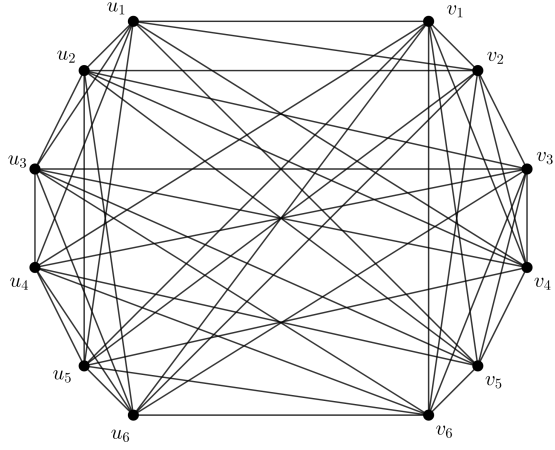}
\caption{Graph $G^{\prime}$ where an eigenvalue has an eigenvector with all entries equal to $\pm 1.$}
\label{Figure6}
\end{figure}
Note that $d_{v_i}(G^{\prime}[V])-d_{v_i}(G^{\prime}[V,U]) =0$ for all vertex $v_{i} \in V$ and  $d_{u_i}(G^{\prime}[U])-d_{u_i}(G^{\prime}[V,U]) =0$ for all vertex $u_{i} \in U.$ From Theorem \ref{Prop14}, $G^{\prime}$ has $p_{S}^{v_i}=1$ and $p_{S}^{u_i}=-1$ for all $i=1, \ldots, 6$ as an eigenvector associated to the eigenvalue $\lambda = 0.$ Let $C_1=u_1u_2v_5v_1$, $C_2=u_1u_2v_4v_3u_3u_4v_6v_5$ and $C_3=u_5u_6v_2v_1$ be cycles of $G^{\prime}$. If we delete the edges of any of those cycles from $G^{\prime}$, $\lambda=0$ is still an eigenvalue of the obtained graph. 
Table \ref{tab:Aeig} presents the deletion of some edges and the correspondent position of the eigenvalue $\lambda = 0$ of $A$ for the obtained graph. In all cases, the eigenvector associated to eigenvalue zero is $p_{S}.$
\begin{table}[h]
\centering
\begin{tabular}{cc}
\hline
List of deleted edges & $i$-th position of the eigenvalue $\lambda_{i} = 0$ \\ \hline 
$\emptyset$ & $\lambda_5=0$ \\ 
$E(C_1)$ &  $\lambda_5=0$ \\ 
$E(C_2)$ &  $\lambda_6=0$ \\ 
$E(C_3)$ &  $\lambda_5=0$ \\ 
$E(C_1)\cup E(C_3)$ &  $\lambda_6=0$ \\ 
$E(C_2)\cup E(C_3)$ &  $\lambda_6=0$\\
\hline
\end{tabular}
\caption{Deletion of some edges of the graph $G$ and the position of the eigenvalue $0$ having $p_{S}$ as an eigenvector.}
\label{tab:Aeig}
\end{table}
Now, let $e_{ij}=v_iv_j$ be the edges of $G[V]$, let $f_{ij}=u_iu_j$ be the edges of $G[U]$ and let $g_{ij}=v_iu_j$ be the edges in $G[E(V,U)].$ Write  $C_4=u_1u_4u_5$, $C_5=u_2u_3u_6$, $C_6=v_1v_2v_3v_4v_5v_6$ and  $C_7=u_1v_1u_4v_6u_5v_2u_6v_3u_3v_4u_2v_5$ as cycles of $G.$ Table \ref{tab:Aeig2} is built by using Theorem \ref{Prop14} and presents the deletion of some edges of $G$ and the correspondent position of the eigenvalues with $p_{S}$ as an eigenvector. 

\begin{table}[h]
\centering
\begin{tabular}{cc}
\hline
List of deleted edges & Eigenvalue of $A$ \\ \hline 
None & $\lambda_5=0$ \\ 
$e_{12},e_{34},e_{56},f_{12},f_{34},f_{56},$ &  $\lambda_8=-1$ \\ 
$E(C_4)\cup E(C_5)\cup E(C_6) $ &  $\lambda_9=-2$ \\ 
$g_{11},g_{22},g_{33},g_{45},g_{54},g_{66}$ &  $\lambda_3=1$ \\ 
$E(C_7)$ &  $\lambda_2=2$\\
\hline
\end{tabular}
\caption{Deletion of some of edges of $G^{\prime}$  and the corresponding eigenvalue having the partition vector $p_{S}$ as an eigenvector}
\label{tab:Aeig2}
\end{table}
\vspace{0.5cm}

\newpage

The following corollary states necessary conditions to a graph $G$ be $L$- or  $Q$-exact graph.

\begin{corollary}\label{Coro} Let $G$ be a graph on $n$ vertices and let $S\subset V$ such that $cut_G(S)= mcut(G)$. Then,
\begin{enumerate}
\item[(i) ] if $G$ is $Q$-exact, the graph $G[S]+G[\overline{S}]$ is $\frac{q_1}{2}$-regular;
\item[(ii) ] if $G$ is $L$-exact, the graph  $G[E(S,\bar{S})]$ is $\frac{\mu_1}{2}$-regular.
\end{enumerate}
\end{corollary}
\begin{proof}
Let $p_S$ be the partition vector associated to the maximum cut $S$. 
Let $G$ be a $Q$-exact graph. From Proposition \ref{Prop3}, $p_{S}$ is an eigenvector to $q_{n}(G).$ So, 
$$\frac{p_S^tQp_S}{n} = q_n(G)=\frac{4}{n}(W-mcut(G)).$$ Therefore, from  Theorem \ref{Prop14}(i), $G[S]+(G-S)$ is $\frac{q_1}{2}$-regular, and the proof of (i) is completed.
Let $G$ be a $L$-exact. From, the Rayleigh-Ritz Theorem, it is clear that $p_{S}$ is an eigenvector to $\mu_{1}(G),$ which implies that 
$$\frac{p_S^tLp_S}{n} = \mu_1(G) =\frac{4}{n}mcut(G).$$ Therefore, according to  Theorem \ref{Prop14}(ii), $G[E(S,\bar{S})]$ is $\frac{\mu_1}{2}$-regular, and the proof follows.

\end{proof}

If $G$ is a non-regular graph, then $A-$, $Q-$ and $L-$ exactness of the graph are mutually exclusive. Thus,  more than one graph exactness may only happen in regular graphs as stated by Corollary \ref{Prop14cor} below.

\begin{corollary}\label{Prop14cor}
Let $G$ be a graph with $n$ vertices and $N_1,N_2,N_3 \in \{A,L,Q\}$ pairwise distinct matrices. Then, $G$ is $N_1$-exact and $N_2$-exact if and only if $G$ is regular and $N_3$-exact.
\end{corollary}
\begin{proof}
The proof follows into three cases:
\begin{enumerate}
\item[(i) ] Let $G$ be $L-$ and $Q$-exact. By Corollary \ref{Coro}, $G[E(S,\overline{S})]$ and  $G[S]+G[\overline{S}]$ are both regular. Therefore, the graph $G$ isomorphic to  $G[E(S,\overline{S})]+(G[S]+G[\overline{S})$ is regular, which implies that $G$ is also $A$-exact.
\item[(ii) ] Let $G$ be $L$-exact and $A$-exact. By  Corollary \ref{Coro} and Theorem \ref{Prop14}, we can assure that $G[E(S,\overline{S})]$ is $s$-regular and $d_i(G[S]+G[\overline{S})-d_i(G[E(S,\overline{S}))=\lambda$ for all $i \in V(G).$ Since 
$$d_i(G) = d_i(G[S]+G[\overline{S})-d_i(G[E(S,\overline{S})]+2d_i(G[E(S,\overline{S})]) = \lambda+2s$$ for each $i=1,\ldots,n$,  we have that $G$ is $(\lambda+2s)$-regular, which implies that $G$ is also $Q$-exact.
\item[(iii) ] Let $G$ be $Q$- and $A$-exact. By  Corollary \ref{Coro}, $G[S]+G[\overline{S}]$ is $r$-regular and $d_{i}(G[S]+G[\overline{S}])-d_{i}(G[E(S,\overline{S})]) = \lambda$ for all $i \in V(G)$. Since $$d_i(G) = d_i(G[E(S,\overline{S})])-d_i(G[S]+G[\overline{S}]) +2 d_i(G[S]+G[\overline{S}])=(2r-\lambda),$$ we get that $G$ is $(2r-\lambda)$-regular. It implies that $G$ is also $L$-exact. 
\end{enumerate}
The reciprocal is immediate and the proof is complete.
\end{proof}

\begin{corollary}\label{Coro2} Let $G$ be a graph on $n$ vertices and let $p_S$ be the partition vector associated to $S\subset V$ such that $cut_G(S) = mcut(G)$. 
\begin{enumerate}
\item[(i) ] If $\,G[E(S,\overline{S})]$ is $r$-regular and $2r\geq\mu_1(G[S]+G[\overline{S}])+\mu_2(G[E(S,\overline{S})]),$ then $G$ is a $L$-exact graph.
\item[(ii) ] If $\,G[S]+G[\overline{S}]$ is $r$-regular and $2r\leq q_n(G[S]+(G-S))+q_{n-1}(G[E(S,\overline{S})]),$ then $G$ is a $Q$-exact graph.
\end{enumerate}
\end{corollary}
\begin{proof}
Suppose that $G[E(S,\overline{S})]$ is $r$-regular and $2r\geq\mu_1(G[S]+G[\overline{S}])+\mu_2(G[E(S,\overline{S})]).$ Let $L_c$ be the Laplacian matrix of $G[E(S,\bar{S})]$ and $L_b$ the Laplacian matrix of $G[S]+G[\overline{S}].$ Since $G[E(S,\overline{S})]$ is $r$-regular, from Theorem \ref{Prop14}(ii), $2r$ is an $L$-eigenvalue of $G$ with eigenvector $p_S.$
\noindent  Assume that $\mu_1(G)  > 2r$ and take $\textbf{u}$ as an eigenvector to $\mu_1(G)$ such that $\|\mu_1\|=1.$  
So,  
\begin{center}
\begin{tabular}{ r c l }
 $\mu_1(G[S]+G[\overline{S}])+\mu_2(G[E(S,\overline{S})])$&$=$&$\max_{\Vert x \Vert =1}x^tL_bx+\max_{\begin{tiny}
 \begin{array}{c}
\Vert x \Vert =1 \\
x^tp_S=0 
 \end{array}
 \end{tiny}}x^tL_cx$\\
&$\geq$&$\textbf{u}^t L_b \textbf{u} + \textbf{u}^t L_c \textbf{u}$ \\
&$=$&$\textbf{u}^t L \textbf{u}$ \\
&$=$&$\mu_1(G) > 2r$ \\
& $\geq$ & $\mu_1(G[S]+G[\overline{S}])+\mu_2(G[E(S,\overline{S})]),$\\
\end{tabular}
\end{center}
which is a contradicition. So, $\mu_1(G)=2r = \frac{4}{n}mcut(G) $ and the result (i) follows.
The proof of (ii) follows by using similar arguments of (i).
\end{proof}

\section{Infinity families of exact graphs}\label{infinityfamilies}

The join operation of the   graphs $G$ and $H$, denoted by $G \bigtriangledown H$, is the graph with vertex set  $V(G \bigtriangledown H)=V(G) \cup V(H)$ and edge set $ E(G \bigtriangledown H) = E(G) \cup E(H) \cup  \{uv:u \in V(G)$ and $v \in V(H)\}.$ In this section we present infinity families of join graphs which are $A$-, $L$- and $Q$-exact graphs. 

\begin{proposition} If $t > 1$ and $G = K_2 \bigtriangledown tK_2$, then $G$ is $Q$-exact but not $L$-exact nor $A$-exact.
\end{proposition}
\begin{proof}
Let $S=\{v_1,v_2\}$ be a subset of $V(G)$ such that $d(v_1)=d(v_2)=2t+1$
and let $p_S$ be the partition vector of $S.$ The graph $G[S]+ G[\overline{S}]$ is a 1-regular graph and from  Theorem  \ref{Prop14}(ii), $2$ is a $Q$-eigenvalue with $p_S$ as an eigenvector.
By Proposition \ref{Prop3}, $G$ is a $Q$-exact graph. Since the graph  $G[E(S,\bar{S})]$ is a non-regular graph, by Corollary \ref{Prop14cor}, we  conclude that $G$ is not  $L$-exact or it is not $A$-exact. It may suggests $G$ can be $A$-exact or $L$-exact. However, from Corollary \ref{Prop14cor}, for $N_1,N_2 \in \{A,L\}$ pairwise distinct matrices, if it is $N_1$-exact and $Q$-exact, then $G$ is $N_2$-exact and regular. That is a contradiction since $G$ is non-regular. Therefore, $G$ is only a $Q$-exact graph.  
\end{proof}

\noindent Next result shows that a join of any two graphs of same order generate a $L$-exact graph.

\begin{proposition} 
\label{prop:Lexact} Let $H_1$ and $H_2$ be two graphs such that $|V(H_1)|=|V(H_2)| = n$. Then $G=H_1\bigtriangledown H_2$ is a $L$-exact graph.
\end{proposition}
\begin{proof}
Let $S = V(H_1)$ and $\overline{S} = V(H_2).$ Note that $G[E(S,\bar{S})]$ is $n$-regular and 
$$2n>n+\frac{n}{2}\geq \max\{\mu_1(H_1),\mu_1(H_2)\}+\mu_2(K_{n,n})=\mu_1(H_1+H_2)+\mu_2(K_{n,n}),$$
where $K_{n,n}$ denotes the complete bipartite graph with both parts of size $n$. The result follows from Corollary \ref{Coro2}(i).
\end{proof}

\begin{proposition}\label{prop:joinQ} Let $G=H_1\bigtriangledown H_2$, where $H_1$ and $H_2$ are $r$-regular graphs of order  $n_1$ and $n_2,$ respectively. If $min\{n_1,n_2\}\geq 2r$, then $G$ is a $Q$-exact graph.
\end{proposition}
\begin{proof}
Take $S=V(H_1).$ The graph $G[E(S,\overline{S})]$ is isomorphic to $K_{n_1,n_2}.$ By hypothesis, 
$$2r \leq q_n(H_1+H_2)+min\{n_1,n_2\}=q_n(H_1+H_2)+q_{n-1}(K_{n_1,n_2})$$
and from Corollary \ref{Coro2}(ii), $G$ is a $Q$-exact graph.
\end{proof}

\noindent \textbf{Remark 4.1} The join of any two $r$-regular graphs does not imply a $Q$-exact graph as in  Proposition \ref{prop:Lexact}. The graph  $K_{3,3}\bigtriangledown K_4$ is not a $Q$-exact graph.

\vspace{0.5cm}

The following result is a straightforward consequence of Proposition \ref{prop:joinQ}.
\begin{corollary}
Let $m,n\geq 4$ be two positive integers. Then, $C_m \bigtriangledown  C_n$ is a $Q$-exact graph.
\end{corollary}

We also know that $K_6=C_3 \bigtriangledown  C_3$ is $Q$-exact graph, but $C_3 \bigtriangledown  C_n$ is not a $Q$-exact graph for any $n \geq 4$.

\begin{proposition}\label{prop:joinA} Let $H_1$ be a $r_1$-regular graph with $n_1$ vertices and $H_2$ be a $r_2$-regular graph with $n_2$ vertices. If $n_1-r_2=n_2-r_1>max\{r_1,r_2\}$, then $G=H_1\bigtriangledown H_2$ is an $A$-exact graph.
\end{proposition}
\begin{proof}
Take $S=V(H_1).$ Thus, $d_i(S)-d_i(S,\overline{S})=r_1-n_2$, for all $i \in S$, and $d_i(\overline{S})-d_i(S,\overline{S})=r_2-n_1=r_1-n_2$, for all $i \in \overline{S}$, then, by Theorem \ref{Prop14} $p_S$ is an eigenvector of $A$ associated to the eigenvalue $r_1-n_2$. Also, since $G$ is obtained by a join of two regular graphs, the characteristic polynomial of $G$ is given by

$$ p_G(x) = \frac{(x-(r_1-n_2))(x-(n_1+r_1))p_{H_1}(x)p_{H_2}(x)}{(x-r_1)(x-r_2)},$$
where $p_{H_i}(x)$ is the characteristic polynomial of the graph $H_i.$ Therefore, $(r_1-n_2)$, $(n_1+r_1)$, the roots of $H_1$ but $r_1$ and the roots of $H_2$ but $r_2$ belong to the spectrum of $G.$  It is well-known that all eigenvalues of $H_i$ are in $[-r_i,r_i]$, for $i=1,2$, then $(r_1-n_2)$ is the least eigenvalue of $A$ since $(r_1-n_2)<-max\{r_1,r_2\}\leq -r_i$, $i=1,2$. It implies that $G$ is an $A$-exact graph and the proof is complete.
\end{proof}

\vspace{0.1cm}

\begin{corollary} Let $H$ be a $r$-regular graph with $n$ vertices. Then $G=H\bigtriangledown (n+r)K_1$ is an $A$-exact graph.
\end{corollary}

\vspace{0.5cm}

\noindent \textbf{Acknowledgements:} The research of Leonardo de Lima was partly funded by the CNPq grant number 314298/2018--5.

\end{document}